\documentclass{amsart}
\usepackage{amssymb,latexsym}
\theoremstyle{plain}
\newtheorem{theorem}{Theorem}
\newtheorem{corollary}{Corollary}
\newtheorem{proposition}{Proposition}
\newtheorem{lemma}{Lemma}
\theoremstyle{definition}

\newtheorem{remark}{Remark}
\newtheorem{question}{Question}

\begin{document}

\title[gonality sequence]
{On the gonality sequence of smooth curves: normalizations of singular curves in a quadric surface}
\author{E. Ballico}
\address{Dept. of Mathematics\\
 University of Trento\\
38123 Povo (TN), Italy}
\email{ballico@science.unitn.it}
\thanks{The author was partially supported by MIUR and GNSAGA of INdAM (Italy).}
\subjclass{14H45; 14H50; 32L10}
\keywords{gonality sequence; smooth curve; nodal curve; quadric surface}

\begin{abstract}
Let $C$ be a smooth curve of genus $g$. For each positive integer $r$ the $r$-gonality
$d_r(C)$ of $C$ is the minimal integer $t$ such that there is $L\in \mbox{Pic}^t(C)$ with
$h^0(C,L) =r+1$. In this paper for all $g\ge 40805$ we construct several examples of smooth curves $C$ of genus $g$
with $d_3(C)/3< d_4(C)/4$, i.e. for which a slope inequality fails.
\end{abstract}

\maketitle

\section{Introduction}\label{S1}
Let $C$ be a smooth and connected projective curve of genus $g\ge 3$. For each integer $r\ge 1$
the $r$-gonality $d_r(C)$ of $C$ is the minimal integer $d$ such that there is a degree $d$
line bundle $L$ on $C$ with $h^0(C,L) \ge r+1$ (\cite{lm}). The sequence $\{d_r(C)\}_{r\ge 1}$
is called the {\it gonality sequence} of $C$. This sequence is important to understand the Brill-Noether
theory of vector bundles on $C$ (\cite{ln1}, \cite{ln2}, \cite{ms}). See \cite{lm}, \S 3, for
general properties of this sequence for an arbitrary curve $C$. For most curves we
have
\begin{equation}\label{eqa1}
\frac{d_r(C)}{r}  \ge \frac{d_{r+1}(C)}{r+1}
\end{equation}
for all $r\ge 2$ (\cite{lm}, Proposition 4.1).
In \cite{lm} H. Lange and G. Martens introduced the following notion. The curve $C$ is said to satisfy
the {\it slope inequality} if (\ref{eqa1}) is satisfied for all $r\ge 2$. Since $d_2(C) \le 2d_1(C)$
for all $C$, the slope inequality is always satisfied for $r=1$. Hence $C$ does not satisfy the slope inequality if and only if there
is at least one integer $r \ge 2$ for which (\ref{eqa1}) fails. Many different examples of such curves
are constructed in \cite{lm}. In this paper we look at the case $r=3$ of (\ref{eqa1}) and prove the following result.

\begin{theorem}\label{i1}
Fix an integer $g\ge 40805$. Then there exists a smooth curve $C$ of genus $g$ such that $d_3(C)/3< d_4(C)/4$.
\end{theorem}

The curves $C$ used to prove Theorem \ref{i1} are the normalization of nodal curves $Y$ contained
in a smooth quadric surface $Q \subset \mathbb {P}^3$. These
families of examples are an extension of \cite{lm}, Example 4.12. We prove
that for the normalization of many of them the rational number
$d_4(C)/4 -d_3(C)/3$ is rather large (Propositions \ref{g5}, \ref{h1} and Corollary \ref{h2}). As an obvious consequence we get
the following statement.

\begin{theorem}\label{i2}
There is a sequence $\{C_g\}_{g\ge 3}$ of smooth curves such that $C_g$ has genus $g$,
$$\lim _{g\to \infty} \frac{d_4(C_g)}{d_3(C_g)}
= 3/2 \ \mbox{and} \ \lim _{g\to \infty} \frac{d_4(C_g)/4 -d_3(C_g)/3}{\sqrt{g}} = 1/12.$$
\end{theorem}

To prove Theorems \ref{i1} and \ref{i2} we need to study the cohomology of certain finite
subsets $S\cup B$ of the smooth quadric surface $Q$. These preliminary lemmas are proved in section \ref{S2}. In section \ref{S3} we use these lemmas in the following way.
Fix an integral nodal curve $Y\in \vert \mathcal {O}_Q(a,a+m)\vert$ and set $S:= \mbox{Sing}(Y)$.
Let $C$ be the normalization of $Y$. Fix any $L\in \mbox{Pic}^z(C)$ evincing $d_4(C)$.
To a general divisor $A\in \vert L\vert$ we associate a set $B\subset Q\setminus S$
such that $\sharp (B)=z$ and $h^1(Q,\mathcal {I}_{S\cup B}(a-2,a+m-2)) >0$. The lemmas
proved in section \ref{S2} show that $z\ge 3a-15$ for a  general $S$ and $m$ not too large, while obviously $d_3(C) \le 2a$. Taking only smooth curves
inside $Q$ we only get a sequence of genera, enough to prove the weaker form of Theorem \ref{i2} with ``~$\limsup$~'' instead of ``~$\lim$~'' (as implicit
in \cite{lm}, Example 4.12).

For all integers $r\ge 2$ and $g\ge 2$ let $\alpha (r,g)$ be the supremum of all rational numbers $d_{r+1}(C)/d_r(C)$ with $C$ a smooth curve of genus $g$.
\begin{question}\label{iq1}
Compute $\alpha ' (r):= \liminf _{g\to \infty } \alpha (r,g)$ and $\alpha ''(r):= \limsup _{g\to \infty} \alpha (r,g)$. Is $\alpha ''(3) = 3/2$~?
\end{question}

We work over an algebraically closed base field with characteristic zero.

I want to thank the referee for several extremely useful remarks (in this version Step ($\diamond$) of the proof of Proposition \ref{g5} is due to the referee).

\section{Preliminaries}\label{S2}
Let $Q\subset \mathbb {P}^3$ be a smooth quadric surface.  For any coherent sheaf $\mathcal {F}$
on $Q$ and any $i\in \mathbb {N}$ set $H^i(\mathcal {F}):= H^i(Q,\mathcal {F})$
and $h^i(\mathcal {F}):= \dim (H^i(\mathcal {F}))$. For all $(a,b)\in  \mathbb {Z}^{2}$
let $\mathcal {O}_Q(a,b)$ denote the line bundle on $Q$ with bidegree $(a,b)$.
We have $h^0(\mathcal {O}_Q(a,b))=(a+1)(b+1)$ and $h^1(\mathcal {O}_Q(a,b))=0$ if $(a,b)\in \mathbb {N}^2$,
while $h^0(\mathcal {O}_Q(a,b)) =0$ if either $a<0$ or $b<0$. If $(a,b)\in \mathbb {N}^2$
and $T\in \vert \mathcal {O}_Q(a,b)\vert$, then we say that $T$ has type $(a,b)$.
The lines contained in $Q$ are the curves $D\subset Q$ with either type $(1,0)$ or
type $(0,1)$.

\begin{remark}\label{g5}
Fix a integer $x>0$ and a general $S\subset Q$ such that $\sharp (S)=x$. Since
$h^0(\mathcal {O}_Q(1,1)) =4$, $h^0(\mathcal {O}_Q(2,1)) =h^0(\mathcal {O}_Q(1,2))=6$ and $S$ is general, we have $\sharp (S\cap T_1)\le 3$ for
every $T_1\in \vert \mathcal {O}_Q(1,1)\vert$, $\sharp (S\cap T_2)\le 5$ for every $T_2\in \vert \mathcal {O}_Q(2,1)\vert$ and $\sharp (S\cap T_3)\le 5$ for every $T_3\in \vert \mathcal {O}_Q(1,2)\vert$.
\end{remark}

\begin{lemma}\label{e2.0}
Fix integers $u>0$ and $v >0$.
Fix a reduced $D\in \vert \mathcal {O}_Q(2,1)\vert$ and a set $S\subset D$
such that $\sharp (S) \le 2v+u+1$, $\sharp (S\cap T) \le 1$ for every line $T\subset D$ (if any)
and $\sharp (S\cap T) \le u+v+1$ for every component $T$ of type $(1,1)$ of $D$ (if any).
Then $h^1(D,\mathcal {I}_{S,D}(u,v))=0$.
\end{lemma}

\begin{proof}
First assume that $D$ is irreducible. Since $D \cong \mathbb {P}^1$ and $\deg (\mathcal {O}_D(u,v)) = 2v+u \ge \sharp (S)-1$, we have $h^1(D,\mathcal {I}_{S,D}(u,v))=0$. Now assume
that $D$ has an irreducible component $A$ of type $(1,1)$ and write $D = A\cup T$
with $T$ a line of type $(1,0)$. Since $\sharp (A\cap S) \le u+v+1$, $A\cong \mathbb {P}^1$ and
$\deg (\mathcal {O}_A(u,v)) = u+v$, we have $h^1(A,\mathcal {I}_{S\cap A,A}(u,v))=0$.
Since $\sharp (T\cap S)\le 1$ and $S\subset D$, we have $\sharp (S\setminus S\cap A) \le 1$.
Since $\deg (T\cap A)=1$, we also have $h^1(T,\mathcal {I}_{(S\setminus S\cap A)\cup (A\cap T),T}(u,v))=0$.
Hence a Mayer-Vietoris exact sequence gives  $h^1(D,\mathcal {I}_{S,D}(u,v))=0$. Now assume that $D$ is the union of $3$ lines
$T_1,T_2,T_3$ with $T_2$ of type $(0,1)$. Since $\sharp (T_i\cap S)\le 1$ for all $i$,
we have $h^1(T_1,\mathcal {I}_{S\cap T_1}(u,v))=0$, $h^1(T_2,\mathcal {I}_{S\cap T_2\setminus
S\cap T_1\cap T_2,T_2}(u-1,v))=0$ and $h^1(T_3,\mathcal {I}_{S\setminus (T_1\cup T_2)\cap S,T_3}(u-1,v-1))=0$. We use two Mayer-Vietoris exact sequences and get first $h^1(T_1\cup T_2,\mathcal {I}_{S\cap (T_1\cup T_2),T_1\cup T_2}(u,v)) =0$ and then $h^1(D,\mathcal {I}_{S,D}(u,v))=0$.
\end{proof}

\begin{lemma}\label{e4}
Fix integers $v \ge u\ge 9$ and set $\alpha := \lfloor u/3\rfloor$. Fix a finite set
$E\subset Q$ such that $\sharp (E) \le v-u +10\alpha  $, no $2$ of the points
of $E$ are contained in a line of $Q$, at most $2u+1$ of the points of $E$ are contained in a curve of type $(1,1)$, at
most $3u+1$ of the points of $E$ are contained in a curve of type $(2,1)$  and at most $3u-4$ of the points of $E$ are
contained in a curve of type $(1,2)$.  Then $h^1(\mathcal {I}_E(u,v)) =0$.
\end{lemma}

\begin{proof}
Notice that $\alpha \ge3$.  Set $\beta := u-3\alpha$

\quad (i) In this step we assume $v=u$.
Set $E_0:=E$.
Take any
$A_1\in
\vert
\mathcal {O}_Q(2,1)\vert$ such that
$a_1:= \sharp (E_0\cap A_1)$ is maximal. Set $F_1:= A_1\cap E_0$ and $E_{1,0}:= E_0\setminus F_1$.
Let $D_1\in \vert \mathcal \mathcal {O}_Q(1,2)\vert$ be a curve such that $b_1:= \sharp (E_{1,0}\cap D_1)$ is maximal. Set $G_1:= E_{1,0}\cap D_1$ and $E_1:= E_{1,0}\setminus G_1$.
For each $i\in \{2,\dots ,\alpha \}$ we define recursively the integers $a_i$ and $ b_i$, the
curves $A_i\in \vert \mathcal {O}_Q(2,1)\vert$, $D_i\in \vert \mathcal {O}_Q(1,2)\vert$ and the sets $F_i$, $E_{i,0}$, $G_i$, $E_i$ in the following way.
Take $A_i\in \vert \mathcal {O}_Q(2,1)\vert$ such that $a_i:= \sharp (A_i\cap E_{i-1})$
is maximal. Set $F_i:= E_{i-1}\cap A_i$ and $E_{i,0}:= E_{i-1}\setminus F_i$.  Take
$D_i\in \vert \mathcal {O}_Q(1,2)\vert$ such that $b_i:= \sharp (D_i\cap E_{i,0})$
is maximal and set $G_i:= D_i\cap E_{i,0}$ and $E_i:= E_{i,0}\setminus G_i$. For each $i\in \{1,\dots ,\alpha \}$ we have the
exact sequences
\begin{align}\label{eqe6}
&0 \to \mathcal {I}_{E_{i,0}}(u-3i+1,u-3i+2) \to \mathcal {I}_{E_{i-1}}(u-3i+3,u-3i+3) \notag \\ 
&\to \mathcal {I}_{F_i,A_i}(u-3i+3,u-3i+3) \to 0
\end{align}
\begin{align}\label{eqe7}
&0 \to \mathcal {I}_{E_i}(u-3i,u-3i) \to \mathcal {I}_{E_{i,0}}(u-3i+1,u-3i+2) \notag \\
& \to \mathcal {I}_{G_i,D_i}(u-3i+1,u-3i+2) \to 0
\end{align}
Notice that the sequences $\{a_i\}_{1 \le i \le \alpha}$ and $\{b_i\}_{1\le i \le \alpha }$ are non-increasing. If $a_i \le 4$, then $E_{i,0} = \emptyset$, because $h^0(Q,\mathcal {O}_Q(2,1)) =6$. If $b_i \le 4$, then $E_i=\emptyset$, because $h^0(Q,\mathcal {O}_Q(1,2)) =6$. Since $\sharp (E) \le 10\alpha$, we get $E_\alpha =\emptyset$.

\quad {\emph {Claim 1:}} For every $i\in \{1,\dots ,\alpha \}$ we have $a_i \le 3u-9i+10$.

\quad {\emph {Proof of Claim 1:}} Assume $a_i \ge 3u-9i+11$. Since at most $3u+1$ of the points
of $E$ are contained in a curve of type $(2,1)$, we have $i\ge 2$. Since the sequences $\{a_n\}$, $\{b_n\}$
are non-increasing and $b_j \ge 5$ if $a_{j+1}>0$, we get
$\sharp (E) \ge i(3u-9i+11) + 5(i-1) = i(3u +16 -9i) -5$. If $i=2$, then we
get $\sharp (E) \ge 6u-9 > 10\alpha $, a contradiction. For any $t\in \mathbb {R}$ set $\phi (t) = t(3u+16-9t)-5$.
The function $\phi$ is increasing in the interval $[0,(3u+16)/18]$ and decreasing if $t\ge (3t+16)/18$. Since $\sharp (E) < \phi (2)$ and $\phi (u/3) =16u/3-5 > \sharp (E)$, we get a contradiction.

\quad {\emph {Claim 2:}} For each $i \in \{1,\dots ,\alpha \}$ we have $h^1(A_i,\mathcal {I}_{F_i,A_i}(u-3i+3,u-3i+3))=0$.

\quad {\emph {Proof of Claim 2:}} By Claim 1 we have $\sharp (F_i) \le 3u-9i+10$. If $A_i$ is irreducible, then Claim 2 is true (e.g. by Lemma \ref{e2.0}). If $A_i$ is the union of $3$ lines,
then $a_i\le 3$ and Claim 2 is true (Lemma \ref{e2.0}). Now assume $A_i = T\cup D$ with
$T$ a smooth conic and $D$ of type $(1,0)$. By Lemma \ref{e2.0} Claim 2 is true
if $\sharp (F_i\cap T) \le 2u-6i+7$. Assume $\sharp (F_i\cap T) \ge 2u-6i+8$. Our assumptions on $E$ imply $i\ge 2$. We have $a_i \ge \sharp (A_i\cap T) \ge 2u-6i+8$. Since $b_j \ge 5$ if $E_{j+1} \ne \emptyset$, we get
$\sharp (E) \ge i(2u-6i+8) +5(i-1) = i(2u+13-6i)-5$. If $i=2$, then $\sharp (E) \ge 4u-3 > 10\alpha$, a contradiction. For every $t\in \mathbb {R}$ set $\psi (t):=
t(2u+13-6t) -5$. Since the function $\psi$ is increasing in the interval $[0,(2u+13)/12]$ and
decreasing for $t>(2u+13)/12$,
$\sharp (E) < \psi (2)$ and $\psi (\alpha ) \ge 13\alpha -5 > 10\alpha$, we get a contradiction.

\quad {\emph {Claim 3:}} For each $i\in \{1,\dots ,\alpha \}$ we have $b_i\le 3u-9i+5$.

\quad {\emph {Proof of Claim 3:}} Assume $b_i \ge 3u-9i+6$. Since $G_i\subseteq E$,
our assumptions on $E$ imply $i \ge 2$. Since $b_i>0$, we
have $a_j\ge 5$ for all $j\le i$. Hence $\sharp (E) \ge 5i + i(3u-9i+6) = i(3u+11 -9i)$. Set
$\tau (t) = t(3u+11-9t)$. The function $\tau (t)$ is increasing in the interval $[0,(3u+11)/18]$ and
decreasing if $t>(3u+11)/18$. Since $\tau (2)  = 6u-14 > \sharp (E)$ and
$\tau (\alpha ) \ge 11\alpha > \sharp (E)$, we get a contradiction.

\quad {\emph {Claim 4:}} For each $i \in \{1,\dots ,\alpha \}$ we have $h^1(D_i,\mathcal {I}_{G_i,D_i}(u-3i+1,u-3i+2))=0$.

\quad {\emph {Proof of Claim 4:}} We apply Lemma \ref{e2.0} taking $D_i\in \vert
\mathcal {O}_Q(1,2)\vert$ instead of an element of $\vert \mathcal {O}_Q(2,1)\vert$. If
$D_i$ is irreducible, then Claim 4 follows from Claim 3, because $D_i\cong \mathbb {P}^1$
and $\deg (\mathcal {O}_{D_i}(u-3i+1,u-3i+2))= 2(u-3i+1)+(u-3i+2)$. If $D_i$ is a union
of $3$ lines, then $b_i\le 3$; in this case we just use that $u-3i+1>0$ and $u-3i+2 > 0$.
Now assume $D_i = T\cup D$ with $T$ a smooth conic and $D$ a line. It is sufficient
to have $\sharp (T\cap G_i) \le 2u-6i+4$ (Lemma \ref{e2.0} for curves of type $(1,2)$). Assume $\sharp (T\cap G_i) \ge 2u-6i+5$. Since $T\cup I \in \vert \mathcal {O}_Q(2,1)\vert$ for all $I\in \vert \mathcal {O}_Q(1,0)\vert$ and $\sharp (T\cap E_{i-1}) \ge 2u-6i+5$, we get $a_i \ge 2u-6i+6$.
Hence $\sharp (E) \ge i(4u-12i+11)$. For any $t\in \mathbb {R}$ set $\eta (t) := t(4u-12t+11)$. The function $\eta (t)$ is increasing in the interval $0 \le t \le (4u+11)/24$ and decreasing
if $t> (4u+11)/24$. Since $\eta (1) = 4u-1 > 10\alpha$ and $\eta (\alpha) = \alpha (4u-12\alpha +11)\ge 11\alpha$, we get a contradiction.

By Claims 2 and 4 and the exact sequences (\ref{eqe6}) and (\ref{eqe7}) we get
$h^1(\mathcal {I}_E(u,u)) \le h^1(\mathcal {I}_{E_{\alpha }}(\beta ,\beta))$.  Since
$E_{\alpha } =\emptyset$,
we have $h^1(\mathcal {I}_{E_{\alpha }}(\beta ,\beta ))=0$.

\quad (ii) Now assume $v>u$. Write $E = F\sqcup F'$ with $\sharp (F') = \min \{\sharp (E), v-u\}$. Since
$\sharp (F')\le v-u$ and no two points of $E$ are contained in an element of $\vert \mathcal {O}_Q(0,1)\vert$, there is
a union $T\subset Q$ of $v-u$ disjoint elements of $\vert \mathcal {O}_Q(0,1)\vert$  such that
$F'\subset T$ and $T\cap F=\emptyset$. We have an exact sequence
\begin{equation}\label{eqe2}
0\to \mathcal {I}_F(u,u)\to \mathcal {I}_E(u,v)\to \mathcal {I}_{F',T}(u,v)\to 0
\end{equation}
Since $T$ is a disjoint union of $v-u$ lines, each of them
containing at most one point of $F'$, we have $h^1(T,\mathcal {I}_{F',T}(u,v))=0$. Step (i) gives $h^1(Q,\mathcal
{I}_F(u,u))=0$.
Hence (\ref{eqe2}) gives $h^1(Q,\mathcal {I}_E(u,v))=0$.
\end{proof}

\begin{lemma}\label{g4}
Fix integers $x, \alpha, \beta ,z$ such that $\alpha \ge 3$, $\beta \ge 2$, $0 \le x\le (\beta +1)^2$ and $0 \le z\le
10\alpha$. Set
$u:= 3\alpha +\beta$. Fix a general
$S\subset Q$ such that $\sharp (S)=x$. Fix $B\subset Q\setminus S$ such that $\sharp (B)
=z$, no line of $Q$ contains
$2$ points of $S\cup B$, $\sharp (B\cap T_1)\le 2u-2$ for
every $T_1\in \vert \mathcal {O}_Q(1,1)\vert$, $\sharp (B\cap T_2)\le 3u-4$ for every $T_2\in \vert \mathcal {O}_Q(2,1)\vert$ and $\sharp (B\cap T_3)\le 3u-9$ for every $T_3\in \vert \mathcal {O}_Q(1,2)\vert$. Then $h^1(\mathcal {I}_{S\cup B}(u,u))
= 0$.
\end{lemma}

\begin{proof}
Set $E_0:= S\cup B$ and $B_0:= B$. Since $S$ is general, we have $\sharp (S\cap T_1)\le 3$ for
every $T_1\in \vert \mathcal {O}_Q(1,1)\vert$, $\sharp (S\cap T_2)\le 5$ for every $T_2\in \vert \mathcal {O}_Q(2,1)\vert$ and $\sharp (S\cap T_3)\le 5$ for every $T_3\in \vert \mathcal {O}_Q(1,2)\vert$ (Remark \ref{g5}). Hence $\sharp
(E_0\cap T_1)\le 2u+1$ for every $T_1\in \vert \mathcal {O}_Q(1,1)\vert$, $\sharp (E_0\cap T_2)\le 3u+1$ for every element
of $\vert \mathcal {O}_Q(2,1)\vert$ and $\sharp (E_0\cap T_3)\le 3u-4$ for every element
of $\vert \mathcal {O}_Q(1,2)\vert$.
Take any
$A_1\in
\vert
\mathcal {O}_Q(2,1)\vert$ such that
$a_1:= \sharp (B_0\cap A_1)$ is maximal. Set $F_1:= A_1\cap E_0$, $B'_1:= A_1\cap B_0$, $E_{1,0}:= E_0\setminus F_1$
and $B_{1,0}:= B_0\setminus B'_1$.
Let $D_1\in \vert \mathcal {O}_Q(1,2)\vert$ be a curve such that $b_1:= \sharp (B_{1,0}\cap D_1)$ is maximal. Set
$G_1:= E_{1,0}\cap D_1$, $B''_1:= B_{1,0}\cap D_1$, $B_1:= B_{1,0}\setminus B''_1$ and $E_1:= E_{1,0}\setminus G_1$. For each
$i\in \{2,\dots ,\alpha \}$ we define recursively the integers $a_i$, $ b_i$, the curves $A_i\in \vert \mathcal
{O}_Q(2,1)\vert$, $D_i\in \vert \mathcal {O}_Q(1,2)\vert$ and the sets $F_i$, $E_{i,0}$, $G_i$, $B'_i$, $B_{i,0}$, $B_i$, $B''_i$, $E_i$ in the
following way. Take
$A_i\in \vert \mathcal {O}_Q(2,1)\vert$ such that $a_i:= \sharp (A_i\cap B_{i-1})$ is maximal. Set $F_i:= E_{i-1}\cap A_i$,
$B'_i:= A_i\cap B_{i-1}$, $B_{i,0}:= B_{i-1} \setminus B'_i$ and
$E_{i,0}:= E_{i-1}\setminus F_i$. Take
$D_i\in \vert \mathcal {O}_Q(1,2)\vert$ such that $b_i:= \sharp (D_i\cap B_{i,0})$
is maximal and set $G_i:= E_{i,0}\cap D_i$, $E_i:= E_{i,0}\setminus G_i$, $B''_i:= B_{i,0}\cap D_i$ and $B_i:= B_{i,0}\setminus B''_i$. For each $i\in \{1,\dots ,\alpha \}$ we have the exact sequences (\ref{eqe6}) and (\ref{eqe7}).
Notice that the sequences $\{a_i\}_{1 \le i \le \alpha}$ and $\{b_i\}_{1\le i \le \alpha }$ are non-increasing. If $a_i \le 4$, then $E_{i,0} = \emptyset$, because $h^0(Q,\mathcal {O}_Q(2,1)) =6$. If $b_i \le 4$, then $E_i=\emptyset$, because $h^0(Q,\mathcal {O}_Q(1,2)) =6$.

\quad {\emph {Claim 1:}} For every $i\in \{1,\dots ,\alpha \}$ we have $a_i \le 3u-9i+5$ and $\sharp (A_i\cap E_{i-1})\le
3u-9i+10$.

\quad {\emph {Proof of Claim 1:}} Since $\sharp (A_i\cap S)\le 5$, it is sufficient
to prove the inequality $a_i\le 3u-9i+5$. Assume $a_i \ge 3u-9i+6$. Since at most
$3u-4$ of the points of $B$ are contained in a curve of type $(2,1)$, we have $i\ge 2$. Since the sequences $\{a_j\}$ and $\{b_j\}$
are non-increasing and $b_j \ge 5$ if $a_{j+1}>0$, we get
$\sharp (B) \ge i(3u-9i+6) + 5(i-1) = i(3u +11 -9i) -5$. If $i=2$, then we
get $\sharp (B) \ge 6u-19 \ge 18\alpha +6\beta -19$, contradicting
the assumptions $\sharp (B) \le 10\alpha $, $\alpha \ge 3$ and $\beta >0$. For any $t\in \mathbb {R}$ set $\phi (t) = t(3u+11-9t)-5$.
The function $\phi$ is increasing in the interval $[0,(3u+11)/18]$ and decreasing if $t\ge (3t+11)/18$. Since $\sharp (B) < \phi (2)$ and $\phi (\alpha ) =\alpha (3u+11-9\alpha) -5  = \alpha (3\beta +11) -5 > 10\alpha \ge \sharp (B)$, we get a contradiction.

\quad {\emph {Claim 2:}} For each $i \in \{1,\dots ,\alpha \}$ we have $h^1(A_i,\mathcal {I}_{F_i,A_i}(u-3i+3,u-3i+3))=0$.

\quad {\emph {Proof of Claim 2:}} By Claim 1 we have $\sharp (F_i) \le 3u-9i+10$. If $F_i$ is irreducible, then Claim 2 is true (e.g. by Lemma \ref{e2.0}). If $A_i$ is the union of $3$ lines,
then $\sharp (F_i)\le 3$ and Claim 2 is true (Lemma \ref{e2.0}). Now assume $A_i = T\cup D$ with
$T$ a smooth conic and $D$ of type $(1,0)$. By Lemma \ref{e2.0} Claim 2 is true
if $\sharp (E_{i-1}\cap T) \le 2u-6i+7$. Assume $\sharp (E_{i-1}\cap T) \ge 2u-6i+8$. Since
$\sharp (S\cap T)\le 3$, we get
$\sharp (B_{i-1}\cap T) \ge 2u-6i+5$. Since $\sharp (T_1\cap B)\le 2u-2$ for every $T_1\in \vert \mathcal {O}_Q(1,1)\vert$, we have $i\ge 2$. We have $a_i \ge \sharp (B_{i-1}\cap T) \ge 2u-6i+5$. Since $b_j \ge 5$ if $B_{j+1} \ne \emptyset$, we get
$\sharp (B) \ge i(2u-6i+5) +5(i-1) = i(2u+10-6i)-5$. If $i=2$, then $\sharp (B) \ge 4u-9$, a contradiction. For every $t\in \mathbb {R}$ set $\psi (t):=
t(2u+10-6t) -5$. The function $\psi$ is increasing in the interval $[0,(u+5)/6]$ and
decreasing for $t>(u+5)/6$.
Since $\sharp (B) < \psi (2)$ and $\psi (\alpha ) = \alpha (2u+10-6\alpha ) =\alpha (10+2\beta )> 10\alpha \ge \sharp (B)$, we get a contradiction.

\quad {\emph {Claim 3:}} For each $i\in \{1,\dots ,\alpha \}$ we have $b_i\le 3u-9i$
and $\sharp (G_i) \le 3u-9i+5$.

\quad {\emph {Proof of Claim 3:}} Since $\sharp (S\cap D_i)\le 5$, it is sufficient
to prove $b_i\le 3u-9i$. Assume $b_i \ge 3u-9i+1$. Since $G_i\subseteq D_i$,
our assumptions on $B$ gives $i \ge 2$. Since $b_i>0$, we
have $a_j\ge 5$ for all $j\le i$. Hence $\sharp (B) \ge 5i + i(3u-9i+1) = i(3u+6 -9i)$. Set
$\tau (t) = t(3u+6-9t)$. The function $\tau (t)$ in increasing in the interval $[0,(3u+6)/18]$ and
decreasing if $t>(3u+6)/18$. Since $\tau (2)  = 6u-24
\ge 18\alpha +3\beta -24 >  10\alpha \ge \sharp (B)$ and
$\tau (\alpha ) = \alpha  (3u+6-9\alpha )
\ge \alpha (6+3\beta ) \ge 12\alpha$, we get a contradiction.

\quad {\emph {Claim 4:}} For each $i \in \{1,\dots ,\alpha \}$ we have $h^1(D_i,\mathcal {I}_{G_i,D_i}(u-3i+1,u-3i+2))=0$.

\quad {\emph {Proof of Claim 4:}} We apply Lemma \ref{e2.0} taking $D_i\in \vert
\mathcal {O}_Q(1,2)\vert$ instead of an element of $\vert \mathcal {O}_Q(2,1)\vert$. If
$D_i$ is irreducible, then Claim 4 follows from Claim 3, because $D_i\cong \mathbb {P}^1$
and $\deg (\mathcal {O}_{D_i}(u-3i+1,u-3i+2))= 2(u-3i+1)+(u-3i+2)$. If $D_i$ is a union
of $3$ lines, then $b_i\le 3$; in this case we just use that $u-3i+1>0$ and $u-3i+2 > 0$.
Now assume $D_i = T\cup D$ with $T$ a smooth conic and $D$ a line. It is sufficient
to have $\sharp (G_i\cap T) \le 2u-6i+4$ (Lemma \ref{e2.0} for curves of type $(1,2)$).
Assume $\sharp (T\cap G_i) \ge 2u-6i+5$. Since $S$ is general and $h^0(\mathcal {O}_Q(1,1))=4$, we have $b_i = \sharp (T\cap B''_i) \ge 2u-6i+2$. Since $T\cup I\in \vert \mathcal {O}_Q(2,1)\vert$ for
each $I\in \vert \mathcal {O}_Q(1,0)\vert$ and $\sharp (T\cap B_{i-1}) \ge 2u-6i+2$, we have $a_i \ge 2u-6i+3$. Hence $a_j\ge 2u-6i+3$ for
all $j\le i$. Hence $\sharp (B) \ge i(4u-12i+5)$. Set $\eta _1(t):= t(4u+5-12t)$. The function $\eta _1(t)$ is increasing if $0 \le t \le  (4u+5)/24$
and decreasing if $t>(4u+5)/24$. We have $\eta _1(1) = 4u-7 = 12\alpha +4\beta -7 > 10\alpha$. We have
$\eta _1(\alpha ) = \alpha (4\beta +5)  \ge 13\alpha > 10\alpha$.
Since $\sharp (B) \le 10 \alpha$, we get a contradiction.

By Claims 2 and 4 and the exact sequences (\ref{eqe6}) and (\ref{eqe7}) we get
$h^1(\mathcal {I}_E(u,u)) \le h^1(\mathcal {I}_{E_{\alpha }}(\beta ,\beta ))$. We have
$E_{\alpha } \subseteq S $. Since $\sharp (S) = x \le (\beta +1)^2$ and $S$ is general,
we have $h^1(\mathcal {I}_S(\beta ,\beta )) =0$. Hence $ h^1(\mathcal {I}_{E_{\alpha }}(\beta ,\beta ))=0$.\end{proof}

\section{$d_4(C)$ for the normalization $C$ of a nodal $Y\subset Q$}\label{S3}

For any finite set $S\subset Q$ let $2S$ denote the first infinitesimal neighborhood of $S$ in $Q$,
i.e. the closed subscheme of $Q$ with $(\mathcal {I}_S)^2$ as its ideal sheaf. The scheme
$2S$ is zero-dimensional, $(2S)_{red}= S$ and $\deg (2S)=3\cdot \sharp (S)$.

\begin{lemma}\label{g1}
Fix integers $a, b, x$ such that $b \ge a \ge 4$ and $0 \le 3x \le ab$. Fix a general
$S\subset Q$ such that $\sharp (S)=x$. We have $h^0(Q,\mathcal {I}_{2S}(a,b))
= (a+1)(b+1) -3x$. Fix a general $Y\in \vert \mathcal {I}_{2S}(a,b)\vert$. Then $Y$ is integral, nodal
and $\mbox{Sing}(Y)=S$.
\end{lemma}

\begin{proof}
We have $h^1(Q,\mathcal {I}_{2S}(a-1,b-1)) =0$ (\cite{bd}, Theorem 1.1). Hence
$h^1(Q,\mathcal {I}_{2S}(a,b)) =0$, i.e. $h^0(Q,\mathcal {I}_{2S}(a,b))
= (a+1)(b+1) -3x$. Since $h^1(Q,\mathcal {I}_{2S}(a-1,b-1)) =0$ and the line bundle
$\mathcal {O}_Q(1,1)$ is very ample, Castelnuovo-Mumford's lemma implies
that the sheaf $\mathcal {I}_{2S}(a,b)$ is spanned. Hence $\vert \mathcal {I}_{2S}(a,b)\vert$ has no base points outside
$S$. Bertini's theorem implies $S = \mbox{Sing}(Y)$. Fix $P\in S$ and set $S':= S\setminus \{P\}$.
Take a general $D\in \vert \mathcal {I}_{2S'}(a-1,b-1)\vert$. Since $h^1(Q,\mathcal {I}_{2S}(a-1,b-1)) =0$, we
have $h^1(Q,\mathcal {I}_{2S'\cup \{P\}}(a-1,b-1)) =0$. Hence $h^0(\mathcal {I}_{2S'\cup \{P\}}(a-1,b-1))
= h^0(\mathcal {I}_{2S'}(a-1,b-1))-1$. Since $D$ is general, we get $P\notin D$. Let $D'\cup D''
\subset Q$ be the reducible conic with $P$ as its singular locus. Since $P\notin D$,  $D\cup D'\cup
D''$ is an element of $\vert \mathcal {I}_{2S}(a,b)\vert$ with an ordinary node at $P$. Since $Y$ is general, it has an
ordinary node at $P$. Since this is true for all $P\in S$, $Y$ is nodal. For every irreducible component $T$ of $Y$ we have $\omega _Q\cdot T = \mathcal {O}_Q(-2,-2) \cdot T<0$. Since $b\ge a \ge 4$ and $3x \le ab$, we have $p_a(Y) = ab-a-b +1 \ge x$. Since $Y$ is nodal, no component of $Y$ appears with multiplicity
$\ge 2$. Since $S$ is general and
$S = \mbox{Sing}(Y)$, the curve $Y$ is irreducible (\cite{ac}, Proposition 4.1).
\end{proof}

\begin{lemma}\label{g2}
Fix integers $a, b, x$ such that $b \ge a \ge 4$ and $0 \le 3x \le (a-1)(b-1)$. Fix a general
$S\subset Q$ such that $\sharp (S)=x$. Fix zero-dimensional schemes $Z, Z'\subset Q$ such that $\deg (Z)=\deg (Z')=2$, $Z_{red}$
and $(Z')_{red}$ are distinct points, $Z$ is contained in a line $D_1\in \vert \mathcal {O}_Q(1,0)\vert$, $Z'$ is contained in a line $D_2\in \vert \mathcal {O}_Q(0,1)\vert$
and $Z\cap D_2=Z'\cap D_1=S\cap (D_1\cup D_2) =\emptyset$. Take a general
$Y\in \vert \mathcal {I}_{Z\cup Z'\cup 2S}(a,b)\vert$. Then 
$h^0(Q,\mathcal {I}_{Z\cup Z'\cup 2S}(a,b)) = (a+1)(b+1)-3x-4$, $Y$ is nodal, integral, 
$\mbox{Sing}(Y)=S$, $\sharp ((Y\cap D_1)_{red}) =b-1$ and $\sharp ((Y\cap D_2)_{red}) =a-1$.
\end{lemma}

\begin{proof}
We have $h^1(Q,\mathcal {I}_{2S}(a-2,b-2)) =0$ (\cite{bd}, Theorem 1.1). We immediately get
$h^1(Q,\mathcal {I}_{Z\cup Z'\cup 2S}(a,b)) = 0$, i.e. $h^0(Q,\mathcal {I}_{Z\cup Z' \cup 2S}(a,b)) = (a+1)(b+1)-3x-4$. We also see that $h^1(Q,\mathcal {I}_{Z\cup Z'\cup 2S}(a-1,b-1)) = 0$.
Hence $\mathcal {I}_{Z\cup Z'\cup 2S}(a,b)$ is spanned by Castenuovo-Mumford's lemma.
Hence $Y$ is smooth outside $S\cup \{Z_{red},(Z')_{red}\}$.
Lemma \ref{g1} applied to the integers $a-1$ and $b-1$ gives the existence of an
integral and nodal curve $T\in \vert \mathcal {I}_{2S}(a-1,b-1)\vert$ such that $S= \mbox{Sing}(T)$. Since
$\mathcal {I}_{2S}(a-1,b-1)$ is spanned, we may find $T$ as above and with $Z_{red}\notin T$ and $(Z')_{red}\notin T$.
Since $T\cup D_1\cup D_2
\in \vert \mathcal {I}_{Z\cup Z' \cup 2S}(a,b)\vert$
and $S\cap (D_1\cup D_2) =\emptyset$, we get that $Y$ is smooth at $(Z)_{red}$ and at $(Z')_{red}$
and nodal at each point of $S$. Since $T$ is irreducible and $Y$ is general, either $Y$ is irreducible or $Y = T_1\cup A_1$ with $T_1\in \vert \mathcal {O}_Q(a,b-1)\vert$
and $A_1\in \vert \mathcal {O}_Q(0,1)\vert$ or $Y = T_2\cup A_2$ with $T_2\in \vert \mathcal {O}_Q(a-1,b)\vert$
and $A_2\in \vert \mathcal {O}_Q(1,0)\vert$ or $Y = T_3\cup A_3$ with $T_3\in \vert \mathcal {O}_Q(a-1,b-1)\vert$
and $A_3\in \vert \mathcal {O}_Q(1,1)\vert$. The last three cases are impossible, because $Y$ is nodal, $\mbox{Sing}(Y) =S$, no $a-1$ of the points of $S$ are contained
in a line and no conic contains $a+b-2$ points of $S$. 

Since $Z\subseteq D_1\cap Y$, we have $\sharp (Y\cap D_1) \le b-1$. Since $\mathcal {I}_{2S\cup Z\cup Z'}(a,b)$ is spanned and $Y$ is general, $Y$ does not
contain the degree $3$ divisor of $D_1$ with $Z_{red}$ as its support. Hence $Z_{red}$ appears with multiplicity two in the scheme
$Y\cap D_1$. We need to prove that the other points of $(Y\cap D_1)_{red}$ appear with multiplicity one in the scheme $Y\cap
D_1$. Fix $P\in D_1\setminus Z_{red}$ and let $W\subset D_1$ be the divisor of degree two with $P$ as its support. Since
$h^1(\mathcal {I}_{2S}(a-1,b-1))=0$ and $b\ge a \ge 3$, we have $h^1(\mathcal {I}_{2S\cup Z\cup Z'\cup W}(a,b))=0$. Hence
$\vert \mathcal {I}_{2S\cup Z\cup Z'\cup W}(a,b)\vert$ has codimension two in $\vert \mathcal {I}_{2S\cup Z\cup
Z'}(a,b)\vert$. Since $\dim (D_1)=1$ and $Y$ is general, $Y$ contains no such scheme $W$. Hence $\sharp ((Y\cap D_1)_{red})
=b-1$. In the same way we prove that
$\sharp ((Y\cap D_2)_{red}) =a-1$.\end{proof}

\begin{lemma}\label{f2}
Fix integers $a, m, x$ such that $a\ge 4$, $0 \le m < a$ and $0 \le 3x \le (a-1)(a+m-1)$. Fix a general $S\subset
Q$ such that $\sharp (S)=x$. Let $C$ be the normalization of a general $Y \in \vert \mathcal {I}_{2S}(a,a+m)\vert$. Let $u_1: C \to \mathbb {P}^1$ (resp. $u_2: C\to \mathbb {P}^1$) be the $g^1_a$ (resp. $g^1_{a+m}$) induced by the projection of $Q$ onto its second (resp. first) factor.
Then neither $u_1$ nor $u_2$ is composed with an involution, i.e. there
are no triple $(C_i,v_i,w_i)$ with $C_i$ a smooth curve, $w_i: C \to C_i$, $v_i: C_i\to \mathbb {P}^1$, $u_i = v_i\circ w_i$, $\deg (v_i)>1$ and $\deg (w_i)>1$.
\end{lemma}

\begin{proof}
It is sufficient to find $O\in \mathbb {P}^1$ and $O'\in \mathbb {P}^1$ such that $u_2^{-1}(O)$
is formed by $a+m-1$ points, one of them being an ordinary ramification point of $u_2$,
and $u_1^{-1}(O')$
is formed by $a-1$ points, one of them being an ordinary ramification point of $u_2$.
Fix zero-dimensional schemes $Z, Z'\subset Q$ such that $\deg (Z)=\deg (Z')=2$, $Z_{red}$
and $(Z')_{red}$ are distinct points, $Z$ is contained in a line $D\in \vert \mathcal {O}_Q(1,0)\vert$, $Z'$ is contained
in a line $D'\in \vert \mathcal {O}_Q(0,1)\vert$ and $S\cap (D\cup D') =\emptyset$. Take a general
$M\in \vert \mathcal {I}_{Z\cup Z'\cup 2S}(a,a+m)\vert$. Lemma \ref{g2} gives that $M$
is a nodal and irreducible curve,  $\mbox{Sing}(M) =S$, $ \sharp ((D\cap M)_{red}) = a+m-1$ and $\sharp ((D'\cap M)_{red}) =a-1$. Let $u': C' \to M$ be the normalization of $M$.
Call $u'_1: C'\to \mathbb {P}^1$ and $u'_2: C'\to \mathbb {P}^1$ the pencils induced by the projections of $Q$. The scheme $M\cap D$ is the disjoint
union of $Z$ and $a+m-2$ distinct points and the scheme $M\cap D'$ is the disjoint
union of $Z'$ and $a-2$ distinct points. Since $(Z')_{red}\cap S = \emptyset$ and $Z_{red}\cap S =\emptyset$, there
are unique points $O_1, O'_1\in C$ such that $u'(O_1) =Z_{red}$ and $u'(O'_1) = (Z')_{red}$. Set $O':= u_2(O'_1)$ and $O:= u_1(O_1)$. The $a-2$ (resp. $a+m-2$) points
of $u_1^{-1}(O')\setminus \{O'_1\}$ (resp. $u_2^{-1}(O)\setminus \{O_1\}$) appear with multiplicity one in the fiber
$u_1^{-1}(O')$, because $u$ is a local isomorphism at each of these points and $D'$ (resp. $D$) is transversal to $M$ outside
$(Z')_{red}$ (resp. $Z_{red}$). Hence $O_1$ (resp. $O'_1$) is an ordinary ramification point of $u'_2$ (resp. $u'_1$). Since
$Y$ is a general equitrivial deformation of $M$ inside $\vert \mathcal {I}_{2S}(a,a+m)\vert$, each $u_i$ has a fiber with a unique ramification
point and this ramification point is an ordinary one.\end{proof}

\begin{proposition}\label{g5}
Fix integers $a, m, x$ such that $a \ge 18$, $x\ge 0$, $0\le m<a$ and 
\begin{equation}\label{eqg00}
x\le a/3 +m
\end{equation}
Fix a general $S\subset Q$ such that $\sharp (S)=x$. A general $Y\in \vert \mathcal {I}_{2S}(a,a+m)\vert$ is an integral
nodal curve with $S$ as its singular locus. Let $u: C \to Y$ denote the normalization map. We have $g(C)=
a^2+am-2a-m+1 -x$ and $d_4(C) \ge 3a-14$.
If $a \ge 4m+43$, then $d_4(C)/4 > d_3(C)/3$.
\end{proposition}

\begin{proof}

Lemma \ref{g2} gives that $Y$ is integral and nodal and that $S =\mbox{Sing}(Y)$. Hence $C$ has genus $a^2+am-2a-m+1-x$.

Since the line bundle $u^\ast (\mathcal {O}_Y(1,1))$ has degree $2a+m$, we have
$d_3(C)\le 2a+m$. Notice that $4(2a+m) < 3(3a-14)$ if $a \ge 4m +43$. Set $z:= d_4(C)$ and assume $z \le 3a-15$.

Fix $L\in \mbox{Pic}^z(C)$ evincing $d_4(C)$. The line bundle $L$ is spanned (\cite{lm}, Lemma 3.1 (b)). Fix a general $A\in \vert L\vert$. Set $B:= u(A)$. Since $L$ has no
base points and $A$ is general, $S\cap B =\emptyset$. Since $Q$ has
only finitely many lines intersecting $S$ and $A$ is general, we may assume
$B$ disjoint from these finitely many lines. Hence no line of $Q$ contains a point of $S$ and at least one point of $B$. 

\quad {\emph {Claim 1:}} $h^1(\mathcal {I}_{S\cup B}(a-2,a+m-2))>0$.

\quad {\emph {Proof of Claim 1:}} Since
$L$ has no base points, we have $h^0(C,\mathcal {O}_C(A\setminus \{O\})) =
h^0(C,\mathcal {O}_C(A)) -1$ for every $O\in A$. Hence
$h^0(C,\omega _C(-A+\{O\})) = h^0(C,\omega _C(-A))$ for every $O\in A$ (Riemann-Roch and Serre
duality). We have $\omega _Q\cong \mathcal {O}_Q(-2,-2)$. Hence
the adjunction formula gives $\omega _Y \cong \mathcal {O}_Y(a-2,a+m-2)$.
Since $h^i(\mathcal {O}_Q(-2,-2)) =0$, $i=0, 1$, the restriction
map $H^0(\mathcal {O}_Q(a-2,a+m-2)) \to H^0(Y,\omega _Y)$ is bijective.
Since $Y$ has only ordinary nodes as singularities, we have $H^0(C,\omega _C) \cong H^0(\mathcal {I}_S(a-2,a+m-2))$. Hence for
any $O\in A$ we have $h^0(\mathcal {I}_{S\cup (B\setminus \{u(O)\})}(a-2,a+m-2))
= h^0(C,\omega _C(-(A\setminus \{O\}))) = h^0(C,\omega _C(-A)) = h^0(\mathcal {I}_{S\cup B}(a-2,a+m-2))$. Hence $h^1(\mathcal {I}_{S\cup B}(a-2,a+m-2))>0$, concluding the proof of Claim 1.

Let $v_2: C \to \mathbb {P}^1$ (resp. $v_1: C \to \mathbb {P}^1$) denote the degree
$a$ (resp. degree $a+m$) morphism obtained composing $u$ with the pencil associated to $\vert \mathcal {O}_Q(0,1)\vert$ (resp. $\vert \mathcal {O}_Q(1,0)\vert$). Lemma \ref{f2} shows that none of these two pencils factors non-trivially. 

\quad {\emph {Claim 2:}} For a general $B$ no line of $Q$ contains two or more points of $B$.

\quad {\emph {Proof of Claim 2:}} Assume for instance that for a general $B$ there is a line $D_B\in \vert \mathcal {O}_Q(0,1)\vert$ such
that $\sharp (D_B\cap B)\ge 2$. Set $\Psi := \{(P,Q)\in C\times C: P\ne Q$ and $v_2(P) =v_2(Q)\}$. For any $D\in \vert \mathcal {O}_Q(0,1)\vert$ the scheme $D\cap Y$
is zero-dimensional. Since $\dim (\vert L\vert )=3 + \dim (\vert \mathcal {O}_Q(0,1)\vert )$,
there is a one-dimensional irreducible set $\Phi \subseteq \Psi$ such that for all $(P,Q)\in \Phi$ the set $\{P,Q\}$ is contained
in a $3$-dimensional family $F_{\{P,Q\}}$ of elements of $\vert L\vert$. Fix $(P,Q)\in \Phi$. Since $L$ has no base points, we have $h^0(L(-P)) =4$.
Hence the existence of the family $F_{\{P,Q\}}$ implies $h^0(L(-P-Q)) = h^0(L(-P))$. Hence $Q$ is a base point of $\vert L(-P)\vert$. The two projections
$C\times C\to C$ induce dominant maps $\Phi \to C$. Hence $P$ may be seen as a general point of $C$. Since $Q\ne P$, the morphism $\varphi : C \to \mathbb {P}^4$
associated to $\vert L\vert$ is not birational onto its image, i.e. $\varphi  = u _2\circ u_1$ with $\deg (u_1)\ge 2$, $u_1: C \to C'$ a morphism of degree $\ge 2$ with $C'$ a smooth curve
and $u_2: C' \to \varphi (C)\hookrightarrow \mathbb {P}^4$ birational onto its image. We have $z = \deg (u_1)\cdot \deg (\varphi (C)) \ge 4 \deg (u_1)$.
Since $u_1(P)=u_1(Q)$ for a general $(P,Q)\in \Phi $, a general fiber of $u_1$ intersects in at least two points a fiber of $v_2$. Since $v_2$ is not composed
with a pencil, $u_1$ factors through $v_2$. Hence $z \ge 4a$, a contradiction. Hence Claim 2 is true.

Since
$S$ is finite, there are only finitely many lines of $Q$ containing at least one point of $S$. Call
$\Gamma$ their union. Since $S$ is general, no such a line contains at least two points of $S$. Since $\vert L\vert$ has no base
points, and $\Gamma \cap Y$ is finite,  for general $B$ we may assume $B\cap \Gamma =\emptyset$. Hence Claim 2 implies that no line of
$Q$ contains at least two points of $S\cup B$. Claim 1 gives $h^1(Q,\mathcal {I}_{S\cup B}(a-2,a+m-2)) >0$. By (\ref{eqg00}) we have $x+z \le 3a-15 +a/3 +m \le
10 \lfloor a/3 \rfloor +20/3 +m -15 \le 10\lfloor a/3\rfloor +m-8$. To apply Lemma
\ref{e4} with $E=S\cup B$, $a=u$ and $v =a+m$ and get a contradiction it is sufficient to prove that $\sharp ((S\cup B)\cap T_1) \le 2a-3$ for all
$T_1\in \vert \mathcal {O}_Q(1,1)\vert$,
$\sharp ((S\cup B)\cap T_2) \le 3(a-2)+1$ for every $T_2\in \vert \mathcal {O}_Q(2,1)\vert$
and $\sharp ((S\cup B)\cap T_2) \le 3(a-2)-4$ for every $T_2\in \vert \mathcal {O}_Q(2,1)\vert$.
Fix $T_1\in \vert \mathcal {O}_Q(1,1)\vert$, $T_2\in \vert \mathcal {O}_Q(2,1)\vert$ and $T_3\in \vert \mathcal {O}_Q(1,2)\vert$.
Set $y_i:= \sharp (S\cap T_i)$ and $a_i:= \sharp (B\cap T_i)$. Assume either
$y_1+a_1 \ge 2a-2$ or $y_2+a_2 \ge 3a-4$ or $y_3+a_3\ge 3a-9$. Since $S$ is general, we have $y_1\le 3$, $y_2 \le 5$ and
$y_3\le 5$. Hence either $a_1\ge 2a-5$ or $a_2\ge 3a-9$ or $a_3 \ge 3a-14$.
Since $z\ge a_i$ for every $i$ such that $T_i$ exists and  $z\le 3a-15$, we get the existence of $T_1\in \vert \mathcal {O}_Q(1,1)\vert$
such that $\sharp (B\cap T_1) \ge 2a-5$. Since any line of $Q$ contains at most
one element of $B$, $T_1$ is irreducible.

\quad Step ($\diamond$) (Proof due to the referee; a simpler form would also prove Claim 2) Let $\varphi : C \to \mathbb {P}^4$ be the morphism defined by $\vert L\vert$. Set $\Gamma := \varphi (C)$. Let $C'$ be the normalization of $\Gamma$
and let $f: C \to C'$ be the covering induced by $\varphi$. Let $g^4_{z'}$, $z' = z/\deg (f)$, be the linear series on $C'$ induced by the inclusion $\Gamma \hookrightarrow \mathbb {P}^4$. Take $A'\in g^4_{z'}$ with $A = f^{-1}(A')$. Since $A\in \vert L\vert$ is general, $A'$ is general. The monodromy group of the general hyperplane section
of $\Gamma$ is the full symmetric group. Hence any $4$ points of a general hyperplane section of $\Gamma$ span a 3-dimensional projective space. Hence for any $E\subset A'$ with $\sharp (E) \ge 4$,  $A'$ is the only element of $g^4_{z'}$ containing $E$.
For any $T\in \vert \mathcal {O}_Q(1,1)\vert$ the set $T\cap Y$ is finite. Since $\dim (\vert \mathcal {O}_Q(1,1)\vert )<4$, there is an infinite family  $\mathcal {F} \subset \vert L\vert$ such that $u(D)$ contains $B\cap T_1$ for all $D\in \mathcal {F}$. Fix $D'\in g^4_{z'}$
such that $D:= f^{-1}(D') \in \mathcal {F}$ and $D' \ne A'$. Since $u(f^{-1}(D'\cap A'))$ contains $B\cap T_1$, there is $P\in A'$ such that $\sharp (u(f^{-1}(P)) \cap (B\cap T_1)) \ge (2a-5)/3$. 
Since $(2a-5)/3 \ge 3$, $T_1$ is the only element of $\vert \mathcal {O}_Q(1,1)\vert$ containing $u(f^{-1}(P))\cap (B\cap T_1)$. Moving $A'$ generally in
$g^4_{z'}(-P)$ we get that $D'$ and $B\cap T_1$ moves into a subscheme $B'$ of $u(f^{-1}(D'))$ contained in some $T'\in \vert \mathcal {O}_Q(1,1)\vert$. Since $P\in D'$
and $\sharp (u(f^{-1}(P)) \cap (B\cap T_1)) \ge 3$, we have $T'=T_1$. Hence $B\cap T_1$ does not move moving $A'$ in $g^4_{z'}(-P)$. Hence $\sharp (f^{-1}(P)) \ge 2a-5$, i.e.
$\deg (f) \ge 2a-5$. Since $z'\ge 4$, we get $z\ge 4(2a-5) $, a contradiction.\end{proof}

\begin{proposition}\label{h1}
Fix integers $x, \alpha, \gamma $ such that $\alpha \ge 3$,
$\gamma \ge 4$, $0 \le x \le (\gamma -1)^2$.
Set $a:= 3\alpha +\gamma$.
Fix a general $S\subset Q$ such that $\sharp (S)=x$ and a general
$Y\in \vert \mathcal {I}_{2S}(a,a)\vert$. $Y$ is integral, nodal and $\mbox{Sing}(Y)=S$.
Let $C$ be the normalization of $Y$. Then $d_4(C) \ge \min \{10\alpha +1,3a-14\}$. 
\end{proposition}

\begin{proof}
Set $z:= d_4(C)$ and assume $z\le 3a-15$ and $z\le 10\alpha$. Take the set-up of the proof of Proposition \ref{g5} with $m =0$. In particular we get
a finite set $B\subset Q\setminus S$ such that $\sharp (B) =z$,
$h^1(Q,\mathcal {I}_{S\cup B}(a-2,a-2)) >0$ and no line of $Q$ contains two points of $S\cup B$.
To get a contradiction we cannot apply Lemma \ref{e4} with $u=v = a-2$ and $E:= S\cup B$, because $x$ may be large. We need
to check that we may apply Lemma \ref{g4} with $u=a-2$ and $\beta =\gamma -2$, i.e. we need to check that no line of $Q$ contains two points of $S\cup B$,
$\sharp (B\cap T_1)\le 2a-6$ for
every $T_1\in \vert \mathcal {O}_Q(1,1)\vert$, $\sharp (B\cap T_2)\le 3a-10$ for every $T_2\in \vert \mathcal {O}_Q(2,1)\vert$ and $\sharp (B\cap T_3)\le 3a-14$ for every $T_3\in \vert \mathcal {O}_Q(1,2)\vert$. Since $z\le 3a-15$, we only need to test the conditions for the lines of $Q$ and that $\sharp (B\cap T_1)\le 2a-6$
for each $T_1\in \vert \mathcal {O}_Q(1,1)\vert$. Step ($\diamond$) of the proof of Proposition \ref{g5} proves the condition for $T_1\in \vert \mathcal {O}_Q(1,1)\vert$. We
may also copy the proof of Claim 2 of the proof of Proposition \ref{g5}, because the assumptions of Lemma \ref{g2} are satisfied.
\end{proof}

\begin{lemma}\label{h4}
Fix integers $a, x$ such that $a \ge 24$ and $0 \le x \le 2a-4$. Fix a general $S\subset Q$ such that $\sharp (S)=x$ and a general
$Y\in \vert \mathcal {I}_{2S}(a,a)\vert$. $Y$ is integral, nodal and $\mbox{Sing}(Y)=S$.
Let $C$ be the normalization of $Y$. Then $2a -5 \le d_3(C) \le 2a$.
\end{lemma}

\begin{proof}
Lemma \ref{g1} gives that $Y$ is integral, nodal and smooth outside $S$. The pull-back of the line bundle $\mathcal {O}_Q(1,1)$
gives $d_3(C) \le 2a$. Assume $z:= d_3(C) \le 2a-6$ and fix $L\in \mbox{Pic}^z(C)$ evincing $d_3(C)$. As in Claims 1 and 2 of the
proof of Proposition \ref{g5} we get a set $B\subset Q\setminus S$ such that $\sharp (B) =z$,
no line of $Q$ contains $2$ of the points of $S\cup B$ and $h^1(\mathcal {I}_{S\cup B}(a-2,a-2))>0$. Since
$z \le 2a-6$ and $a\ge 16$, we have $z+5 \le 3a-15$. Hence $\sharp (T\cap (S\cup B)) \le 3a-15$ for
every $T\in \vert \mathcal {O}_Q(2,1)\vert$ and every $T\in \vert \mathcal {O}_Q(1,2)\vert$. Since $z\le 2a-6$, we
have $\sharp (B\cap T) \le 2(a-2)-2$ for every $T\in \vert \mathcal {O}_Q(1,1)\vert$. Set $\gamma := \max \{2,-1+\sqrt{x}\}$, $\alpha := \lfloor (a-2-\gamma )/3\rfloor$
and $\beta := a-2 -3\alpha$. We have $\beta \ge \gamma \ge 2$ and $x \le (\gamma +1)^2 \le (\beta +1)^2$. To apply Lemma \ref{g4} (and hence to get
 $h^1(\mathcal {I}_{S\cup B}(a-2,a-2))=0$, i.e. a contradiction) it is sufficient to prove that $z\le 10\alpha$. This is true if $x \le 8$, because in this
 case $\gamma =2$. Hence we may assume $\gamma = -1+\sqrt{x}$. Hence $\alpha \ge (a-4 -\sqrt{2a})/3$. We have $10(a-4-\sqrt{2a})/3 \ge 2a-6$
 if and only if $4a-22 \ge 10\sqrt{2a}$. Hence it is sufficient to assume $a\ge 24$.\end{proof}

\begin{remark}\label{h3}
Take $C$ as in Proposition \ref{h1}. Set $M:= u^\ast (\mathcal {O}_Y(2,1))\in \mbox{Pic}^{3a}(C)$. Since $h^0(Q,\mathcal {O}_Q(2,1)) =6$ and $Y$ is contained in no element of $\vert \mathcal {O}_Q(2,1)\vert$, we
have $h^0(M)\ge 6$. Fix $P\in Y$ such that $P\in \mbox{Sing}(Y)$ if $x>0$.
Let $F\subset C$ be the scheme-theoretic pull-back of the scheme $P$. We have $\deg (F) \ge 1+\min \{1,x\}$. Since $h^0(Q,\mathcal {I}_E(2,1)) =4$,
we have $h^0(M(-F)) \ge 5$.
Hence $d_4(C) \le 3a - 1 -\min \{1,x\}$.
\end{remark}

\begin{corollary}\label{h2}
Fix integers $a, x$ such that $a\ge 204$ and $0 \le x \le 2a-4$. Fix a general $S\subset Q$ such that $\sharp (S)=x$ and a general
$Y\in \vert \mathcal {I}_{2S}(a,a)\vert$. $Y$ is integral, nodal and $\mbox{Sing}(Y)=S$.
Let $C$ be the normalization of $Y$. Then $2a-5 \le d_3(C) \le 2a$ and $3a-15
\le d_4(C) \le 3a -1-\min \{1,x\}$.
\end{corollary}

\begin{proof}
Since $3x \le 6a-12 \le a^2$, $Y$ is integral, nodal and $\mbox{Sing}(Y) = S$
(Lemma \ref{g1}). Lemma \ref{h4} gives $2a-5 \le d_3(C) \le 2a$. Remark \ref{h3} gives $d_4(C) \le 3a-1-\min \{x,1\}$. Assume $z:= d_4(C) \le 3a-16$.  Take $B$ as in the proofs of Proposition \ref{g5} and \ref{h1}. We have $h^1(\mathcal {I}_{S\cup B}(a-2,a-2)) >0$. Set $\delta := -1
+\lceil  \sqrt{2a-4}\rceil$ and $\alpha := \lfloor (a-2 -\delta )/3\rfloor$. Notice
that $x \le 2a-4 \le (\delta +1)^2$ and that $\alpha \ge (a-5-\sqrt{2a})/3$. Since $a\ge 204$, we have $a-2 \ge 10\sqrt{2a}$,
i.e. $10(a-5-\sqrt{2a})/3 \ge 3a-16$. Hence $10\alpha \ge 3a-16$. By assumption we have
$d_3(C) \le 3(a-2)-9$. Claim 2 and Step ($\diamond$) of the proof of Proposition \ref{g5}
show that we may apply Lemma \ref{g4} with the integers $u:=a-2$ and $\beta:= a-2-3\alpha$ (notice that $\beta \ge \delta$) and get $h^1(\mathcal {I}_{S\cup B}(a-2,a-2)) =0$, a contradiction.
\end{proof}

\vspace{0.3cm}
\qquad {\emph {Proof of Theorem \ref{i1}.}} For all integers $a, x$ set $g_a:= a^2-2a+1$ and $g_{a,x} = g_a-x$. If $a>0$, then $g_a = p_a(Y)$  for any $Y\in\vert \mathcal {O}_Q(a,a)\vert$. Hence
if $Y$ is a nodal curve of type $(a,a)$ with exactly $x$ nodes, then $g_{a,x}$ is the genus of the normalization of $Y$. Now assume $a\ge 2$. We have $g_a-g_{a-1} = a^2-2a+1-a^2+2a-1+2a-2-1 =2a-3$. Hence the set $\{g_{a,x}\}_{0 \le x \le 2a-4}$ contains every integer
between $g_{a-1}+1$ and $g_a$. We take the set-up of the proof of Corollary
\ref{h2}. Fix an integer $g  \ge 40805$. Let $a$ be the only integer
such that $g_{a-1} < g \le g_a$. Since $g_{203} =40804$, we have $a\ge 204$. We have $g = g_a -x$ with $0\le x\le 2a-4$. Apply Corollary \ref{h2}.\qed

Of course, the lower bound $g\ge 40805$ is not sharp.

\vspace{0.3cm}

\qquad {\emph {Proof of Theorem \ref{i2}.}} For any $g< 40805$ we take as $C_g$ an arbitrary smooth
curve of genus $g$. Fix an integer $g \ge 40805$ and call $a$ the minimal positive integer such that
$g \le a^2-2a+1$. Set $x:= a^2-2a+1-g$. Since $g> (a-1)^2-2(a-1)+1$, we have $x \le 2a-4$.
Fix a general $S\subset Q$ such that $\sharp (S) =x$. Take as $C_g$ the normalization of
a general $Y\in \vert \mathcal {I}_{2S}(a,a)\vert$. Corollary \ref{h2} gives $2a-6 \le d_3(C) \le 2a$
and
$3a -15\le d_4(C) \le 3a -1 -\min \{1,x\}$. We have $ g_{a-1} < g \le g_a$. Hence the limits
are as in the statement of Theorem \ref{i2}.\qed

\providecommand{\bysame}{\leavevmode\hbox to3em{\hrulefill}\thinspace}

\end{document}